\documentclass[12pt]{article}

\usepackage{latexsym}
\usepackage[centertags]{amsmath}
\usepackage{amsfonts}
\usepackage{amssymb}
\usepackage{amsthm}
\usepackage{newlfont}
\usepackage{graphics}
\usepackage{color}

\newtheorem{thm}{Theorem}
\newtheorem{cor}[thm]{Corollary}
\newtheorem{lem}[thm]{Lemma}
\newtheorem{prop}[thm]{Proposition}

\newtheorem{conj}[thm]{Conjecture}
\newtheorem{iden}[thm]{Identity}

\def\CT{\mathop{\mathrm{CT}}}
\allowdisplaybreaks[1]

\def\QQ{\mathbb{Q}}
\def\NN{\mathbb{N}}
\def\PP{\mathbb{P}}

\title{On Zeilberger's Constant Term \\
for Andrews' TSSCPP Theorem}

\author{Guoce Xin\thanks{The author would like to thank Doron Zeilberger for suggesting this subject, and thank the referee for valuable suggestions improving this exposition. Part of this work was done during the author's stay at the Center for Combinatorics, Nankai University. This work was supported by the Natural Science Foundation of China.}\\
\small Department of Mathematics\\[-0.8ex]
\small Capital Normal University, Beijing 100048, PR China\\
\small \texttt{guoce.xin@gmail.com}\\
}

\date{June 14, 2011\\
\small Mathematics Subject Classifications: 05A15, 05A19}

\begin{document}
\maketitle

\begin{flushright}
  Dedicated to Doron Zeilberger's 60th birthday
\end{flushright}

\vspace{5mm}

\begin{abstract}
This paper studies Zeilberger's two prized constant term identities. For one of the identities, Zeilberger asked for a simple proof that may give rise to a simple proof of Andrews theorem for  the number of totally symmetric self complementary plane partitions. We obtain an identity reducing a constant term in $2k$ variables to a constant term in $k$ variables. As applications,  Zeilberger's constant terms are converted to single determinants. The result extends for two classes of matrices, the sum of all of whose full rank minors is converted to a single determinant. One of the prized constant term problems is solved, and we give a seemingly new approach to Macdonald's constant term for root system of type BC.
\end{abstract}

%

\section{Introduction}

In 1986 \cite{MRR}, Mills, Robbins and Rumsey defined a class of objects called \emph{totally symmetric self complementary plane partitions} (denoted TSSCPP for short) and conjectured that the number $t_n$ of TSSCPPs of order $n$ is given by
\begin{align}
  \label{e-tn}
t_n= A_n:=\prod_{i=0}^{n-1} \frac{(3i+1)!}{(n+i)!},
\end{align}
which also counts the number of alternating sign matrices, a famous combinatorial structure, of order $n$.
In 1994, Andrews \cite{Andrews} proved the conjecture by using Stembridge's Pfaffian representation \cite{Stem} derived from Doran's combinatorial characterization \cite{Doron} of $t_n$. At the same time, Zeilberger suggested a constant term approach in \cite{zeilberger-dollars}, as we describe below.

We only need Doran's description of $t_n$ in \cite{Doron}:  $t_n$ equals the sum of all the $n\times n$ minors of the $n\times (2n-1)$ matrix $\left( \displaystyle\binom{i-1}{j-i} \right)_{1\le i \le n, 1\le j \le 2n-1}$. The sum can be transformed to a constant term
by simple algebra manipulation. Thus, combining equation \eqref{e-tn}, we can obtain the following  identity:
\begin{iden}\label{iden-And}
$$\CT_x \frac{\prod_{1\le i< j \le n}
(1-\frac{x_i}{x_j})\prod_{i=1}^n(1+x_i^{-1})^{i-1}}{\prod_{i=1}^n
(1-x_i) \prod_{1\le i<j\le n} (1-x_ix_j) } =\prod_{i=0}^{n-1}
\frac{(3i+1)!}{(n+i)!}.
$$\end{iden}
Zeilberger observed that a simple proof of this identity will give rise to a simple proof of Andrews' TSSCPP theorem. He offered a prize asking for a direct constant term proof. A prize is also offered for the following identity.
\begin{iden}\label{iden-simple}
$$\frac{1}{n!}\CT_x \frac{\prod_{1\le i\ne j \le n}
(1-\frac{x_i}{x_j})\prod_{i=1}^n(1+x_i^{-1})^m}{\prod_{i=1}^n
(1-x_i) \prod_{1\le i<j\le n} (1-x_ix_j) } =\prod_{j=0}^{n-1}
\prod_{i=1}^m \frac{2i+j}{i+j}.
$$\end{iden}

In 2007, I had a chance to meet Doron Zeilberger and to discuss the advantage of using partial fraction decomposition and the theory of iterated Laurent series in dealing with the $q$-Dyson related problems. See, e.g., \cite{qDyson1,qDyson2}. Thereafter
he suggested that I shall consider the above two identities. In this paper, only Identity \ref{iden-simple} is given a direct constant term proof. In addition, a conjecture is given as a generalization of Identity \ref{iden-And}.

The paper is organized as follows. Section 1 is this introduction. Section 2 includes the main results of this paper. By using partial fraction decomposition, we derive a constant term reduction identity that reduces a constant term in $2k$ variables to a constant term in $k$ variables. Applications are given in Section 3. For two classes of matrices, the sum of all full rank minors are converted to a single determinant. We also make a conjecture generalizing  Identitie \ref{iden-And}. Section 4 completes the proof of Identity \ref{iden-simple}. We also include a method to evaluate Macdonald's constant term for root system of type BC.

\section{Constant term reduction identities}
In this paper, we only need to work in the ring of Laurent series
$\QQ((x_1,x_2,\dots,x_n))$. For $\pi \in \mathfrak{S}_n$ we use the usual notation
$\pi f(x_1,x_2,\dots,x_n):=f(x_{\pi_1},x_{\pi_2},\dots,x_{\pi_n})$. The easy but useful SS-trick (short for Stanton-Stembridge trick) states that if $f\in \QQ((x_1,x_2,\dots,x_n))$, then
$$ \CT_x f(x_1,x_2,\dots,x_n) = \frac{1}{n!} \CT_x \sum_{\pi\in \mathfrak{S}_n} \pi f(x_1,x_2,\dots,x_n).$$
See, e.g., \cite[p. 9]{SS-trick}. We will often use the SS-trick without mentioning.

\def\sgn{\textrm{sgn}}

We need some notations. Define
\begin{align}
\label{e-Bd} B_k(x)&:=\det \left(x_i^{-j}-x_i^{j}
\right)_{1\le i,j \le k}=\sum_{\pi\in \mathfrak{S}_k} \sgn(\pi) \pi(x_1^{-1}-x_1)\cdots (x_k^{-k}-x_k^k),\\
\bar{B}_k(x)&:=\det \left(x_i^{j-1}+x_i^{-j} \right)_{1\le i,j \le
k}=\sum_{\pi\in \mathfrak{S}_k} \sgn(\pi) \pi(1+x_1^{-1})\cdots
(x_k^{k-1}+x_k^{-k}).\label{e-Bbd}
\end{align}
Then it is well-known that
\begin{align} B_k &=\prod_{1\le i\le k}\frac{1-x_i^2}{x_i^k}\prod_{1\le i<j\le
k}({x_i}-{x_j})(1-x_ix_j),\label{e-Bp}\\
\bar{B}_k &=\prod_{1\le i\le k}\frac{1+x_i}{x_i^k}\prod_{1\le i<j\le
k}({x_i}-{x_j})(1-x_ix_j).\label{e-Bbp}
\end{align}

A rational function $Q$ is said to be \emph{gratifying} in
$x_1,x_2,\dots,x_n$ if we can write
\begin{align}
\label{e-Q}Q=Q(x_1,\dots,x_n)= \frac{\prod_{1\le i\ne j \le n}
(1-\frac{x_i}{x_j})P(x_1^{-1},\dots, x_n^{-1})}{\prod_{i=1}^n
(1-x_i) \prod_{1\le i<j\le n} (1-x_ix_j) },
\end{align}  where
$P(x_1,\dots,x_n)$ is a polynomial.

Now we can state our main result as the following. The proof will be given later.
\begin{thm}\label{t-P}
Let $Q$ be as in \eqref{e-Q} with $P$ a symmetric polynomial. If
$n=2k$, then
\begin{align}
\CT_x  Q(x)&=\frac{(2k)!}{2^k}\CT_x
P(x_1,\dots,x_k,x_{1}^{-1},\dots,x_k^{-1})
\bar{B}_k(x)\prod_{i=1}^k(x_i^{i}+x_i^{1-i}) \label{e-Qeven}\\
&=(2k-1)!!\CT_x P(x_1,\dots,x_k,x_{1}^{-1},\dots,x_k^{-1})
\bar{B}_k(x)^2\prod_{i=1}^k x_i;\tag{\ref{e-Qeven}$'$}
\label{e-Qeven'}
\end{align}
if $n=2k+1$, then
\begin{align}\label{e-Qodd}
\CT_x  Q(x)&=\frac{(2k+1)!}{(-2)^k}\CT_x
P(x_1,\dots,x_k,x_{1}^{-1},\dots,x_k^{-1},1) B_k(x)\prod_{i=1}^k
(x_i^{-i}-x_i^i) \\
&=(-1)^k(2k+1)!!\CT_x P(x_1,\dots,x_k,x_{1}^{-1},\dots,x_k^{-1},1)
B_k(x)^2 \tag{\ref{e-Qodd}$'$}\label{e-Qodd'}.
\end{align}
\end{thm}
Note that the operator $\CT_x$ is valid since $\CT_{x_i} F = F$ if $F$ is free of $x_i$.
We give the following nice form as a consequence.
\begin{cor}\label{c-p}
Let $p(z)$ be a univariate polynomial in $z$. If $n=2k$ then
\begin{align}
\label{e-c-pe} \frac{1}{n!}\CT_x \frac{\prod_{1\le i\ne j \le n}
(1-\frac{x_i}{x_j})\prod_{i=1}^{n} p(x_i^{-1})}{\prod_{i=1}^{n}
(1-x_i) \prod_{1\le i<j\le n} (1-x_ix_j) } &=\CT_x
\bar{B}_k(x)\prod_{i=1}^k x_i^i p(x_i)p(x_i^{-1});
\end{align}
If $n=2k+1$ then \begin{align} \label{e-c-po} \frac{1}{n!}\CT_x
\frac{\prod_{1\le i\ne j \le n} (1-\frac{x_i}{x_j})\prod_{i=1}^{n}
p(x_i^{-1})}{\prod_{i=1}^{n} (1-x_i) \prod_{1\le i<j\le n}
(1-x_ix_j) } &=p(1)\CT_x {B}_k(x)\prod_{i=1}^k x_i^i
p(x_i)p(x_i^{-1}) .
\end{align}
\end{cor}
\begin{proof}
By applying Theorem \ref{t-P} with $P(x)=\prod_{i=1}^{2k}p(x_i)$,
the left-hand side of \eqref{e-c-pe} becomes
\begin{align*}
 & \CT_x {2^{-k}}\det \left(x_i^{j-1}+x_i^{-j} \right)_{1\le i,j \le
k} \prod_{i=1}^k(x_i^{i}+x_i^{1-i})
\prod_{i=1}^kp(x_i)p(x_i^{-1})\\
 &={2^{-k}}\det \left(\CT_{x_i}(x_i^{j-1}+x_i^{-j})(x_i^{i}+x_i^{1-i})p(x_i)p(x_i^{-1})  \right)_{1\le i,j \le
 k} \\
 &=\det \left( \CT_{x_i} (x_i^{i+j-1}+x_i^{i-j})p(x_i)p(x_i^{-1}) \right)_{1\le i,j \le
 k}\\
&=\CT_x \bar{B}_k(x)\prod_{i=1}^k x_i^i p(x_i)p(x_i^{-1}).
\end{align*}
Here we used the fact $\CT_x x^i p(x) p(x^{-1})= \CT_x x^{-i} p(x) p(x^{-1})$. Similarly,
by applying Theorem \ref{t-P} with $P(x)=\prod_{i=1}^{2k+1}p(x_i)\rule{0pt}{16pt}$,
the left-hand side of \eqref{e-c-po} becomes
\begin{align*}
 & p(1)\CT_x {(-2)^{-k}}\det \left(x_i^{-j}-x_i^{j} \right)_{1\le i,j \le
k}  \prod_{i=1}^k(x_i^{-i}-x_i^{i})
\prod_{i=1}^kp(x_i)p(x_i^{-1})\\
 &=p(1){(-2)^{-k}}\det \left(\CT_{x_i}(x_i^{-j}-x_i^{j})(x_i^{-i}-x_i^{i})p(x_i)p(x_i^{-1})  \right)_{1\le i,j \le
 k} \\
 &=p(1)\det \left( \CT_{x_i} (x_i^{i-j}-x_i^{i+j})p(x_i)p(x_i^{-1}) \right)_{1\le i,j \le
 k}\\
&=p(1)\CT_x {B}_k(x)\prod_{i=1}^k x_i^i p(x_i)p(x_i^{-1}).
\end{align*}
\end{proof}

In order to prove Theorem \ref{t-P}, we need some notations.
The degree $\deg_{x_1}Q$ of a rational function $Q$ in $x_1$ is
defined to be the degree of the numerator minus the degree of the
denominator in $x_1$. If $\deg_{x_1}Q<0$, then we say that $Q$ is
\emph{proper} in $x_1$. The partial fraction decomposition of a
proper rational function has no polynomial part. The following lemma is by
direct application of partial fraction decomposition.

\begin{lem}\label{l-gratifying}
If $Q$ is gratifying in $x_1,x_2,\dots,x_n$, then
$$ \CT_{x_1} Q(x_1,\dots,x_n) = A_0+A_2+\cdots+A_n, $$
where $A_0=Q(1-x_1) \big|_{x_1=1}$,
$A_r=Q(1-x_1x_r)\big|_{x_1=1/x_r}, 2\le r\le n$. Moreover, $A_0$ is
gratifying in $x_2,\dots,x_n$, and $A_r$ is gratifying in
$x_2,\dots,x_{r-1},x_{r+1},\dots,x_n$.
\end{lem}
\begin{proof}
Assume $Q$ is given by \eqref{e-Q}. We claim that $Q$ is proper in
$x_1$. This can be easily checked by observing that for $m$ being free of $x_1$,
the degree (in $x_1$) of $(1-x_1 m)$ is $1$ and the degree of $1-m/x_1$ and $1-m$ are both $0$.

Now the partial fraction decomposition of $Q$ can be written in the
following form.
\begin{align*}
Q = \frac{p_0(x_1)}{x_1^d} + \frac{A_0}{1-x_1}+\sum_{r=2}^n
\frac{A_r}{1-x_1x_r},
\end{align*}
where $d$ is a nonnegative integer, $p_0(x_1)$ is a polynomial of
degree less than $d$, and $A_0,A_2,\dots, A_n$ are independent of
$x_1$
 given by $A_0=Q(x)(1-x_1)
\big|_{x_1=1}$, $A_r=Q(x)(1-x_1x_r) \big|_{x_1=x_r^{-1}}$ for $r\ge
2$. Now clearly we have
$$\CT_{x_1} Q(x)= A_0+A_2+\cdots +A_n. $$
This proves the first part of the lemma.

For the second part, we need to rewrite $A_r$ in the right form. For
$r=0$ we have
\begin{align*}
A_0 &=\left.\frac{\prod_{j=2}^n (1-\frac{x_1}{x_j})
\prod_{i=2}^n(1-\frac{x_i}{x_1})   \prod_{2\le i\ne j \le n}
(1-\frac{x_i}{x_j})P(x_1^{-1},\dots, x_n^{-1})}{\prod_{i=2}^n
(1-x_i) \prod_{j=2}^n(1-x_1x_j)\prod_{2\le i<j\le n}
(1-x_ix_j) }\right|_{x_1=1}\\
&=\frac{\prod_{j=2}^n (1-\frac{1}{x_j}) \prod_{i=2}^n(1-{x_i})
\prod_{2\le i\ne j \le n} (1-\frac{x_i}{x_j})P(1,x_2^{-1},\dots,
x_n^{-1})}{\prod_{i=2}^{n} (1-x_i)
\prod_{j=2}^{n}(1-{x_j})\prod_{2\le i<j\le n} (1-x_ix_j) } \\
&= \frac{\prod_{2\le i\ne j \le n} (1-\frac{x_i}{x_j})
P'(x_2^{-1},\dots,x_n^{-1})}{\prod_{i=2}^{n} (1-x_i)\prod_{2\le
i<j\le n} (1-x_ix_j) },
\end{align*}
where $P'(x_2,\dots,x_n)$ is a polynomial in $x_2,\dots,x_n$ given
by
$$ P'(x_2^{-1},\dots,x_n^{-1})=P(1,x_2^{-1},\dots,
x_n^{-1}) \prod_{i=2}^{n} (1-x_i^{-1}).$$ Thus $A_0$ is gratifying
in $x_2,\dots,x_n$ as desired.

For $r\ge 2$, without loss of generality, we may assume $r=n$. We
have
\begin{align*}
A_n &=\left.\frac{\prod_{j=2}^n (1-\frac{x_1}{x_j})
\prod_{i=2}^n(1-\frac{x_i}{x_1})   \prod_{2\le i\ne j \le n}
(1-\frac{x_i}{x_j})P(x_1^{-1},\dots, x_n^{-1})}{(1-x_1)\prod_{i=2}^n
(1-x_i) \prod_{j=2,j\ne n}^n(1-x_1x_j)\prod_{2\le i<j\le n}
(1-x_ix_j) }\right|_{x_1=1/x_n}\\
&=\frac{\prod_{j=2}^n (1-\frac{1}{x_nx_j}) \prod_{i=2}^n(1-{x_ix_n})
\prod_{2\le i\ne j \le n} (1-\frac{x_i}{x_j})P(x_n,x_2^{-1},\dots,
x_n^{-1})}{(1-\frac{1}{x_n})\prod_{i=2}^{n} (1-x_i)
\prod_{j=2}^{n-1}(1-\frac{x_j}{x_n})\prod_{2\le i<j\le n}
(1-x_ix_j) }
\end{align*}

After massive cancelation, we obtain
\begin{align*}
A_n &=\frac{P''(x_2^{-1},\dots,x_{n-1}^{-1})\prod_{2\le i\ne j \le
n-1} (1-\frac{x_i}{x_j}) }{\prod_{i=2}^{n-1} (1-x_i) \prod_{2\le
i<j\le n-1} (1-x_ix_j)},
\end{align*}
where $P''(x_2,\dots,x_{n-1})$ is a polynomial in $x_2,\dots,
x_{n-1}$ given by
\begin{align*}
  \frac{P''(x_2^{-1},\dots,x_{n-1}^{-1})}{P(x_n,x_2^{-1},\dots, x_n^{-1})}
  &=\frac{(1-\frac{1}{x_n^2})(1-x_n^2)\prod_{j=2}^{n-1} (1-\frac{1}{x_nx_j})
\prod_{j=2}^{n-1} (1-\frac{x_n}{x_j})}{(1-\frac{1}{x_n})(1-x_n)
 } \\
  &=\frac{(1+x_n)^2}{x_n}
\prod_{j=2}^{n-1} (1-\frac{1}{x_nx_j})(1-\frac{x_n}{x_j}).
\end{align*}
Thus $A_n$ is gratifying in $x_2,\dots,x_{n-1}$ as desired.
\end{proof}

To evaluate the constant term of a gratifying $Q$, we can
iteratively apply Lemma \ref{l-gratifying}. This will result in a
big sum of simple terms. We shall associate to each term a partial
matching to keep track of them. To be precise, we describe this as
follows.

Start with $Q$ associated with the empty matching. At every step we
have a set of terms, each associated with a partial matching
consisting of blocks of size $1$ or $2$. For a term $R$ associated
with $M$, we can see from iterative application of Lemma
\ref{l-gratifying} that $R$ is gratifying in all variables except
for those with indices in $M$. If $M$ is a full matching, i.e., of
$[n]:=\{1,2,\dots,n\}$, then put $R$ into the output; otherwise suppose the smallest
such variable is $x_i$. Then applying Lemma \ref{l-gratifying} with
respect to $x_i$ gives a sum of terms. One term is similar to $A_0$,
associate to it with $M \cup \{\{i\}\}$, and the other terms are
similar to $A_r$, associate to it $M\cup \{\{i, r\}\}$.

If we denote by $Q_M$ the term corresponding to $M$, then we have
$$Q_M= Q(1-x_{i_1}x_{j_1})\cdots(1-x_{i_s}x_{j_s})(1-x_{i_{s+1}})\cdots(1-x_{i_{s+r}})
\Big|_{x_{i_e}=x_{j_e}^{-1},x_{i_f}=1}^{1\le e\le s<f\le s+r} , $$
where $\{i_e,j_e\}$ and $\{i_f\}$ are all the 2-blocks and 1-blocks.

Observing that in the $A_0$-terms the factor $(1-x_j)$ appears in the
numerator, we see that $Q_M=0$ if $M$ has two singleton blocks.

The above argument actually gives the following result.

\begin{prop}\label{p-Matching}
If $Q$ is gratifying in $x_1,\dots,x_n$, then
  $$ \CT_x Q =\sum_{M} \CT_x Q_M,  $$
  where the sum ranges over all full matchings with at most one
  singleton block.
\end{prop}
This result becomes nice when $Q$ is symmetric. We need the following lemma, which is by straightforward
calculation.
\begin{lem}\label{l-simplify}
Let $Q$ be as in \eqref{e-Q} with $P=1$. If $n=2k$, then we have
\begin{align}
Q_{\{\{1,k+1\},\dots,\{k,2k\} \}} 
&=\bar{B}_k(x_{k+1},\dots,x_{2k})^2x_{k+1}x_{k+2}\cdots x_{2k};
\end{align}
If $n=2k+1$, then we have
\begin{align}
Q_{\{\{1,k+1\},\dots,\{k,2k\},\{2k+1\} \}}
&=(-1)^kB_k(x_{k+1},\dots,x_{2k})^2.
\end{align}
\end{lem}
Note that we have the following alternative expressions:
\begin{align*} Q_{\{\{1,k+1\},\dots,\{k,2k\}
\}}&=\bar{B}_k(x_{k+1},\dots,x_{2k})\bar{B}_k(x_{k+1}^{-1},\dots,x_{2k}^{-1}),\\
Q_{\{\{1,k+1\},\dots,\{k,2k\},\{2k+1\} \}}
&=B_k(x_{k+1},\dots,x_{2k})B_k(x_{k+1}^{-1},\dots,x_{2k}^{-1}).
\end{align*}

\begin{proof}[Proof of Theorem \ref{t-P}]

If $n=2k$, then Proposition \ref{p-Matching} states that
$$\CT_x Q = \sum_{M} \CT_x Q_M, $$
where $M$ ranges over all complete matchings of $[n]$, i.e., every
block has exactly two elements. There are
$(2k-1)!!=(2k-1)(2k-3)\cdots 1$ such $M$. Since $Q_M$ are all
Laurent series and $Q$ is symmetric in all variables, they have the
same constant terms. Therefore
\begin{align}
\CT_x Q= (2k-1)!!\CT_{x_{k+1},\dots, x_{2k}} Q_{M_0}, \label{e-ct-Q}
\end{align}
 where $M_0$ is taken to be
$\{\{1,k+1\},\dots,\{k,2k\} \}$. It is an exercise to show that
\begin{align*}
Q_{M_0}=P(x_{k+1},\dots,x_{2k},x_{k+1}^{-1},\dots,x_{2k}^{-1})\bar{B}_k(x_{k+1},\dots,x_{2k})^2x_{k+1}\cdots
x_{2k}.
\end{align*}
This gives \eqref{e-Qeven'} immediately after renaming the parameters. By applying the SS-trick, Lemma \ref{l-simplify}, and equation \eqref{e-Bbd}, we obtain
$$\CT_{x}Q_{M_0} =k!\CT_{x} P(x_1,\dots,x_{k}, x_1^{-1},\dots,x_k^{-1}) \bar{B}_k  \prod_{i=1}^k(x_i^{i}+x_i^{1-i}).  $$
The above formula and \eqref{e-ct-Q} yield \eqref{e-Qeven}.

If $n=2k+1$, then by a similar argument, we have
$$\CT_x Q =(2k+1)!!\CT_{x_{k+1},\dots, x_{2k}} Q_{M_1},  $$
where $M_1$ is taken to be ${\{\{1,k+1\},\dots,\{k,2k\},\{2k+1\}
\}}$ and we have
\begin{align*}
Q_{M_1}=(-1)^kP(x_{k+1},\dots,x_{2k},x_{k+1}^{-1},\dots,x_{2k}^{-1},1)B_k(x_{k+1},\dots,x_{2k})^2.
\end{align*}
Thus \eqref{e-Qodd} and \eqref{e-Qodd'} follow in a similar way.
\end{proof}

\section{Applications: a determinants reduction identity}
Zeilberger obtained the following more general transformation in \cite{zeilberger-dollars}.
\begin{thm}[Zeilberger]\label{t-fg}
Let $f(x)$ and $g(x)$ be polynomials and let $M$ be the $n\times ((\deg f)+(n-1) \deg(g)+1)$ matrix with entries given by
$$ M_{i,j} = \CT_x  \frac{f(x)g(x)^{i-1}}{x^{j-1}}.$$
Then the sum of all $n\times n$ minors of $M$ equals
\begin{align}\label{e-fg}
   \frac{1}{n!}\CT_x \frac{ \prod_{i=1}^n f(x_i^{-1}) \prod_{1\le i<j\le n} (x_i-x_j)(g(x_i^{-1})-g(x_j^{-1}))}
{\prod_{i=1}^n (1-x_i) \prod_{1\le i<j\le n} (1-x_ix_j)}.
 \end{align}
\end{thm}
He considered two cases: i) $g(x)=x(1+x)$, and ii) $g(x)=1+x$, both with $f(x)=(1+x)^{m}$. Case i) with $m=0$ corresponds to Identity \ref{iden-And} and Case ii) corresponds to Identity \ref{iden-simple}.

We start with Case ii), which is easier to simplify. Observe that
$$  (x_i-x_j)(g(x_i^{-1})-g(x_j^{-1}))
=  (x_i-x_j)(x_i^{-1}-x_j^{-1}) = (1-x_i/x_j)(1-x_j/x_i).$$
Then by applying Corollary \ref{c-p} with $p(x)=f(x)$, we obtain:

\begin{thm}\label{t-red-det}
  Let $M$ be as in Theorem \ref{t-fg} with $g(x)=1+x$. Then the sum of all $n\times n$ minors of $M$ equals
  $$ \left\{\begin{array}
    {ll}
    \displaystyle \det \left( \CT_{x} (x^{i+j-1}+x^{i-j})f(x)f(x^{-1}) \right)_{1\le i,j \le
 k}
   & \textrm{ if } n=2k;\\
    \displaystyle  f(1)\det \left( \CT_{x} (x^{i-j}-x^{i+j})f(x)f(x^{-1}) \right)_{1\le i,j \le
 k}   & \textrm{ if } n=2k+1. \rule{0pt}{16pt}
  \end{array}  \right.  $$
\end{thm}

In particular, when $f(x)=(1+x)^m$, the left hand side of \eqref{iden-simple} becomes
$$\left\{\begin{array}
    {ll}
    \displaystyle \det \left( \binom{2m}{m+1-i-j}+\binom{2m}{m-i+j}\right)_{1\le i,j \le
 k}
   & \textrm{ if } n=2k;\\
    \displaystyle  2^m \det \left( \binom{2m}{m-i+j}-\binom{2m}{m-i-j}\right)_{1\le i,j \le
 k}  & \textrm{ if } n=2k+1.
  \end{array}  \right.$$
These determinants should be easy to evaluate$^1$\!, but Zeilberger prefered to
avoid using ``determinants" technique. This leads to the proof in Section \ref{sec-Jacobi}.

\vspace{5mm}
Case i) is a little complicated. One can summarize a formula as in Theorem \ref{t-red-det}, but we will assume $f(x)=(1+x)^m$ for brevity. Note that in \cite{zeilberger-dollars}, the exponent $i-1$ for $1+x_i^{-1}$
was correct in the proof, but was replaced by the wrong exponent
$n-i$ in the formula for $C$.

Denote \eqref{e-fg} with $f(x)=(1+x)^m$ and $g(x)=x(1+x)$ by $LHS$. We have
\begin{align*}
LHS=&\frac{1}{n!}\CT_x
\frac{\prod_{r=1}^n(1+x_r^{-1})^m\prod_{1\le i< j \le n}
(x_i-x_j)
(x_i^{-1}(1+x_i^{-1})-x_j^{-1}(1+x_j^{-1}))}{\prod_{i=1}^n (1-x_i)
\prod_{1\le i<j\le n} (1-x_ix_j) }\\
=&\frac{1}{n!}\CT_x \frac{\prod_{1\le i<
j \le n}(x_i-x_j)(x_i^{-1}-x_j^{-1})
}{\prod_{i=1}^n
(1-x_i) \prod_{1\le i<j\le n} (1-x_ix_j) } P(x^{-1}),
\end{align*}
where $P(x^{-1})$ given by
$$P(x^{-1})=\prod_{r=1}^n(1+x_r^{-1})^m\prod_{1\le i< j \le n}
(x_i^{-1}-x_j^{-1})^{-1}(x_i^{-1}+x_i^{-2}-x_j^{-1}-x_j^{-2})$$
is symmetric in the $x$'s.

By noticing $(x_i-x_j)(x_i^{-1}-x_j^{-1}) = (1-x_i/x_j)(1-x_j/x_i)$, we shall apply Theorem \ref{t-P} with
 $P$ given above. Let us consider the $n=2k$ case first. For clarity we use $g(x_i^{-1})$ for $x_i^{-1}+x_i^{-2} $. By using \eqref{e-Qeven'} and dividing the product for $1\le i<j \le n$ into the following three parts: i) $1\le i<j\le k$, ii) $k+1\le i<j\le 2k$, iii) $1\le i\le k<j\le 2k$, and then splitting part iii) as $i=j-k$, $i<j-k$, and $i>j-k$,  we have
\begin{align*}
LHS&=\frac{1}{2^k k!} \CT_x \bar B_k(x)^2 \prod_{i=1}^k x_i \prod_{r=1}^n(1+x_r^{-1})^m\prod_{1\le i< j \le n}
\frac{g(x_i^{-1})-g(x_j^{-1})}{x_i^{-1}-x_j^{-1}} \Big|_{x_{k+\ell}=x_\ell}^{\ell=1,\ldots,k}\\
&=\frac{1}{2^k k!}\CT_x \bar B_k(x)\bar B_k(x^{-1})\prod_{r=1}^k(1+x_r^{-1})^m(1+x_r)^m \prod_{i=1}^k\frac{g(x_i^{-1})-g(x_i)}{x_i^{-1}-x_i}\\
&\quad\times \prod_{1\le i< j \le k}\frac{(g(x_i^{-1})-g(x_j^{-1}))( g(x_i)-g(x_j))(g(x_i^{-1})-g(x_j))(g(x_i)-x_j^{-1}) }{
(x_i^{-1}-x_j^{-1})(x_{i}-x_{j}) (x_i^{-1}-x_{j})(x_i-x_j^{-1})}\\
&=\frac{1}{2^k k!}\CT_x \frac{\bar B_k(x)\bar B_k(x^{-1}) \prod_{1\le i\le k}(1+x_i)^{2m}(x_i^{-1}+1+x_i)\prod_{1\le i< j \le k}U_{i,j}}{\prod_{1\le i\le k} x_i^m\prod_{1\le i< j \le k}
(x_i^{-1}-x_j^{-1})(x_{i}-x_{j}) (1-x_ix_j)(1-x_i^{-1}x_j^{-1})}\\
&= \frac{1}{2^k k!}\CT_x\prod_{1\le i\le k}\frac{(x_i^{-1}+1+x_i)(1+x_i)^{2m+2}}{x_i^{m+1}} \prod_{1\le i< j \le k}{U_{i,j} }
\end{align*}
where $U_{i,j}$ is given by
\begin{align*}
  U_{i,j}&=(g(x_i^{-1})-g(x_j^{-1}))( g(x_i)-g(x_j))(g(x_i^{-1})-g(x_j))(g(x_i)-g(x_j^{-1})).
\end{align*}
 Since $U_{i,j}$ is invariant under
replacing $x_i$ by $x_i^{-1}$ or $x_j$ by $x_j^{-1}$, we can write
it in terms of $z_i$ and $z_j$ where $z_r=x_r+2+x_r^{-1}=x_r^{-1}(1+x_r)^2$:
\begin{equation*}
U_{i,j}=\left(
1-3\,z_{{i}}z_{{j}}+z_{{i}}{z_{{j}}}^{2}+{z_{{i}}}^{2}z_{{j}}
 \right)  \left(z_i -z_{{j}}\right) ^{2}.
\end{equation*}
A crucial observation is that we can write
\begin{align}
  \label{e-Uijy}
  U_{i,j}=z_iz_j (z_i^{-1}(z_i-1)^3-z_j^{-1}(z_j-1)^3)(z_i-z_j).
\end{align}

 Thus
\begin{align*}LHS&=\frac{1}{2^kk!} \CT_x \prod_{1\le i\le k}
z_i^{m+1}(z_i-1)\prod_{1\le i< j \le k} U_{i,j}
\\
&=\frac{1}{2^kk!} \CT_x  \prod_{1\le i\le k}
z_i^{m+k}(z_i-1)\prod_{1\le i< j \le k}
(z_i^{-1}(z_i-1)^3-z_j^{-1}(z_j-1)^3)(z_i-z_j)\\
&=\frac{1}{2^k} \CT_x \prod_{1\le i\le k} z_i^{m+k}(z_i-1)
z_i^{-(i-1)}(z_i-1)^{3(i-1)}\prod_{1\le i< j \le k} (z_i-z_j)
\end{align*}
Therefore we have the following determinant representation.
\begin{align}
  LHS &= \frac{1}{2^k} \det \begin{pmatrix}
                        \CT_x z^{m+k+j-i}(z-1)^{3i-2} \\
                      \end{pmatrix}_{1\le i,j\le k} \label{e-Am-2k}\\
&= \frac{1}{2^k} \det \begin{pmatrix}
                        \CT_x (x^{-1}(1+x)^2)^{m+k+j-i}(x+1+x^{-1})^{3i-2} \\
                     \end{pmatrix}_{1\le i,j\le k}\nonumber
\end{align}

The $n=2k+1$ case is very similar. We only have the extra factor
$$2^m \prod_{i=1}^k (x_i+x_i^2-2)(x_i^{-1}+x_i^{-2}-2) =2^m\prod_{i=1}^k
(2z_i+1)(z_i-4).$$
We have, similarly by the use of \eqref{e-Uijy},
\begin{align}
  LHS &= \frac{2^m}{2^k} \det \begin{pmatrix}
                        \CT_x z^{m+k+j-i}(z-1)^{3i-2}(2z+1)(z-4) \\
                      \end{pmatrix}_{1\le i,j\le k} \label{e-Am-2k1}
\end{align}
The two determinants in $(\ref{e-Am-2k},\ref{e-Am-2k1})$ might be easy for experts by ``determinants" techniques. Here we only make the following conjecture.

\begin{conj}\!\!\!$^2$
Let $M$ be the $n\times (2n+m-1)$ matrix with entries given by
$$M_{i,j}=\binom{m+i-1}{j-i},\  1\le i \le n , 1\le j \le 2n+m-1,$$
Then the sum of all $n\times n$ minors of $M$ equals
  $$\left\{\begin{array}
    {ll}
    \displaystyle \prod_{i=1}^k{\frac {(2i-2)! \left( 2\,i+2\,m-1 \right)! \left(
3\,m+4\,i-2
 \right)_{2i-2}\left( 3\,m+4\,i\right)_{2i-1} }{ \left( m+4\,i-4
 \right) !\, \left( m+4\,i-2 \right) ! }},
   & \textrm{ if } n=2k;\\
    \displaystyle  2^m\prod_{i=1}^k\frac{(2i-1)!(2m+2i+3)!
(3m+4i)_{2i-1}(3m+4i+2)_{2i}}{(m+4i-2)!(m+4i)!(2m+2i+1)_3},  & \textrm{ if } n=2k+1.
  \end{array}  \right.$$
\end{conj}
Here $(n)_k$ is the rising factorial $n(n+1)\cdots (n+k-1)$.

\section{By Jacobi's Change of Variable Formula\label{sec-Jacobi}}
We first complete the proof of Identity \ref{iden-simple} by transforming the constant term into known constant terms. Here, we mean
Macdonald's constant terms for root system of type BC, which is defined to be the constant term of the following:
\begin{multline}
M_n(x;a,b,c):=\prod_{1\le i\le n}(1-x_i)^a\left(1-\frac{1}{x_i}\right)^a
(1+x_i)^b\left(1+\frac{1}{x_i}\right)^b  \\
\prod_{1\le i<j\le n} \left[
\left(1-\frac{x_i}{x_j}\right)\left(1-\frac{x_j}{x_i}\right)(1-x_ix_j)\left(1-\frac{1}{x_ix_j}\right)
\right]^c. \label{e-Macdonald}
\end{multline}
This includes type D (set $a=b=0$), C (set $b=0$), B (set $a=b$) as special cases.
The constant term was evaluated by Macdonald \cite{Macdonald}.

\begin{proof}[Proof of Identity \ref{iden-simple}]
Denote by $LHS$ the left-hand side of Identity \ref{iden-simple}. Apply Theorem
\ref{t-P} with $P(x)=(1+x_1)^m\cdots (1+x_n)^m$.

i) If $n=2k$, then by \eqref{e-Qeven'} we have the following.
\begin{align*}
LHS&=\frac{1}{2^kk!}\CT_x \prod_{i=1}^k (1+x_i)^m(1+x_i^{-1})^m
\bar{B}_k(x)^2\prod_{i=1}^k x_i\\
&= \frac{1}{2^kk!} \CT_x  M_k(x; 0,m+1,1);
\end{align*}

ii) If $n=2k+1$ then similarly by \eqref{e-Qodd'} we have
\begin{align*}
LHS&=\frac{1}{(2)^k k!}\CT_x 2^m (-1)^k \prod_{i=1}^k (1+x_i)^m(1+x_i^{-1})^m
B_k(x)^2 \\
&=\frac{2^{m-k}}{ k!} \CT_x M_k(x; 1,m+1,1).
\end{align*}

The remaining part is routine. We omit the details.
\end{proof}

Before realizing Macdonald's constant term identity applies, we discovered a different approach. This leads to a new way, as far as I know, to evaluate Macdonald's constant term $M_n(x;a,b,c)$ for root system of type BC by using two well-known results. One result is Jacobi's change of
variable formula. See, e.g., \cite{xin-residue}.
\begin{thm}[Jacobi's Residue Formula] Let \(y=f(x)\in \mathbb{C}((x)) \) be a Laurent
series and let \(b\) be the integer such that \(f(x)/x^b\) is a
formal power series with nonzero constant term. Then for any formal
series \(G(y)\) such that the composition \(G(f(x))\) is a Laurent
series, we have \vspace{-2mm}
\begin{align}
 \CT_x G(f(x)) \frac{x}{f}\frac{\partial
f}{\partial x} =b \CT_y G(y). \label{e-Jacobi}
\end{align}
\end{thm}

The other result is the well-known Morris constant term
identity \cite{morris82}.
\begin{thm}
[Morris Identity] For $k\in \PP,$ $b\in \NN$, $a\in \mathbb{C}$, we
have
\begin{align}
\CT_x \prod_{l=1}^n \left(1-{x_l}\right)^{a}
\left(1-\frac{1}{x_l}\right)^{\!\!b}\prod_{1\leq i\neq j\leq n}
  \left( 1-\frac{x_i}{x_j} \right)^{\!\!k}=\prod_{l=0}^{n-1}\frac{(a+b+kl)!(k(l+1))!}{(a+kl)!(b+kl)!k!}.
\end{align}
\end{thm}

We  make the change of variable by \(y_i=x_i(1+x_i)^{-2}\)
with \(b=1\). Then $x_i$ has to be chosen
to be $x_i=\frac{1-2y_i-\sqrt{1-4y_i}}{2y_i}$, which is the well-known
Catalan generating function (minus $1$). Direct calculation shows
that
$$ \frac{x_i}{y_i}\frac{\partial
y_i}{\partial x_i}=\frac{1-x_i}{1+x_i}=\sqrt{1-4y_i}.  $$ Thus
Jacobi's formula gives
\begin{align}
\CT_{x_i} G(y_i(x_i)) =\CT_{x_i}
G(y_i(x_i))\frac1{\sqrt{1-4y_i}}\frac{1-x_i}{1+x_i} = \CT_{y_i}
\frac{G(y_i)}{\sqrt{1-4y_i}}. \label{e-crucial-jacobi}
\end{align}
We also need the following crucial observation.
\begin{align}
\frac{(y_i-y_j)^2}{y_i^2y_j^2}=
(1-\frac{x_i}{x_j})(1-\frac{x_j}{x_i})(1-x_ix_j)(1-\frac{1}{x_ix_j}).
\label{e-crucial-observation}
\end{align}

Now we can compute as follows.
\begin{align*}
 \CT_x M_n(x;a,b,c)&=\CT_x \prod_{1\le i\le n}\left(\frac{1-4y_i}{y_i}\right)^{a}
\left(\frac{1}{y_i}\right)^b   \prod_{1\le i<j\le n} \left[
\frac{(y_i-y_j)^2}{y_i^2y_j^2} \right]^c \\
&=\CT_y \prod_{1\le i\le n}\left({1-4y_i}\right)^{a-1/2}
\left({y_i}\right)^{-a-b-(n-1)c}   \prod_{1\le i\ne j\le n}
\left(1-\frac{y_i}{y_j}\right)^c
\end{align*}
Now make another change of variables by letting $y_i=t_i/4$. We have
\begin{align*}
 \CT_x M_n(x;a,b,c)&=4^{n(a+b+(n-1)c)}\CT_t \prod_{1\le i\le n}\left({1-t_i}\right)^{a-1/2}
\left({t_i}\right)^{-a-b-(n-1)c}   \prod_{1\le i\ne j\le n}
\left(1-\frac{t_i}{t_j}\right)^c
\end{align*}
This corresponds to the Morris identity for parameters $-\frac12-b-(n-1)c,
a+b+(n-1)c, c. $


\medskip
After the paper was published, I received two comments from Professor Christian Krattenhaler. I would like to include them here.

$^1$ The determinants on page 8 are special cases of Theorems 30
and 31 in ``Advanced Determinant Calculus" by Krattenhaler, S\'{e}minaire Lotharingien Combin. 42 (1999), Article B42q.

$^2$ Conjecture 10 is already stated in Section 5, in an equivalent form, of Zeilbergers's paper \cite{zeilberger-dollars}! Moreover, the conjecture has been
proven by Krattenhaler using Stembridge's Pfaffian representation
in ``Determinant identities and a generalization of the number of totally symmetric self-complementary plane partitions", (Electron. J. Combin. 4(1) (1997), \#R27.).

\end{document}